\documentclass[]{preprint}
\usepackage{times}%
\usepackage{amsmath}
\usepackage{amssymb}%
\usepackage{ScriptFonts}%
\usepackage{mhequ}
\usepackage{mhenvs}
\usepackage{mhsymb}
\usepackage{mhfig}
\usepackage{microtype}

\usepackage{tikz}
\usetikzlibrary{snakes}
\usetikzlibrary{decorations}
\usetikzlibrary{decorations.pathreplacing}
\usetikzlibrary{backgrounds,fit}
\pgfrealjobname{revision_pt1_jqw}
\long\def\beginpgfgraphicnamed#1#2\endpgfgraphicnamed{\includegraphics{#1}}

\usepackage{graphicx}
\usepackage{verbatim}

\newtheorem{innercustomthm}{Algorithm}
\newenvironment{algorithm}[1]
  {\renewcommand\theinnercustomthm{#1}\innercustomthm}
  {\endinnercustomthm}

\definecolor{darkred}{rgb}{0.9,0.1,0.1}

\def\E{\mathbf{E}}

\def\v{\chi}

\begin{document} 

  \title{Improved diffusion Monte Carlo}
  \author{Martin~Hairer\inst{1} and Jonathan~Weare\inst{2}}
  \institute{ Mathematics Department, the University of Warwick  \and Department of Statistics and James Franck Institute, the University of Chicago
   \\ \email{M.Hairer@Warwick.ac.uk}, \email{weare@uchicago.edu}}
  \titleindent=0.65cm

  \maketitle
  \thispagestyle{empty}

\begin{abstract}
We propose a modification, based on the RESTART (repetitive simulation trials after reaching thresholds) and DPR (dynamics probability redistribution) rare event simulation algorithms, 
of the standard diffusion Monte Carlo (DMC) algorithm. The new algorithm has a lower variance
per workload, regardless of the regime considered. In particular, it makes it feasible to
use DMC in situations where the ``na\"\i ve'' generalisation of the standard algorithm would
be impractical, due to an exponential explosion of its variance.
We numerically demonstrate the effectiveness of the new algorithm on a standard rare event simulation problem
(probability of an unlikely transition in a Lennard-Jones cluster), as well as a high-frequency
data assimilation problem. 
\end{abstract}


\tableofcontents

\section{Introduction}
Diffusion Monte Carlo (DMC) is a well established method popularized within the Quantum Monte Carlo (QMC) community  to compute the ground state energy (the lowest eigenvalue) of the Hamiltonian operator
\[
\mathcal{H} \psi =  -\Delta \psi  + V \psi
\]
as well as averages with respect to the square of  the corresponding eigenfunction 
\cite{Kalos1962,GrimmStorer1971,Anderson1975,CeperleyAlder1980,KolorencL2011}.
It is based on the fact that the Feynman--Kac formula
\begin{equation}\label{avet}
\int f(x) \psi(x,t) dx =  \mathbf{E}\left( f(B_t) \exp\Bigl({-\int_0^t V(B_s)ds}\Bigr)\right)\;,
\end{equation}
connects the solution, $\psi(x,t)$, of the partial differential equation
\begin{align}\label{psipde}
&\partial_t \psi =  -\mathcal{H}\psi \notag\\
& \psi(0,\cdot) = \delta_0\;,
\end{align} 
to expectations of a standard Brownian motion $B_t$ starting at $x$.  
For large times $t$, suitably normalised integrals against the solution of \eqref{psipde} 
approximate normalised integrals against the ground state eigenfunction of $\mathcal{H}$.



As we will see in the next section, DMC is an extremely flexible tool with many potential applications 
going well beyond quantum Monte Carlo.   In fact, the basic operational principles of DMC, as described in this article, can be traced back at least to what would today be called a rare event simulation scheme introduced in \cite{HammersleySIS:1954} and \cite{RosenbluthSIS:1955}.  In the data assimilation context a popular variant of DMC, sequential importance sampling (see e.g. \cite{defreitas05, liu02}), computes normalised expectations of the form
\[
\frac{\mathbf{E}\left( f(y_t) \exp\left( -\sum_{t_k\leq t} V(k,y_{t_k})\right)\right)}{\mathbf{E} \exp\left( -\sum_{t_k\leq t} V(k,y_{t_k})\right) }\;,
\]
where the $V(k,\cdot)$ are functions encoding, for example, the likelihood of  sequentially arriving observations (at times $t_1<t_2<t_3<\cdots$) given the state of some Markov process $y_t$.  
A central component of a sequential importance sampling scheme is a so-called ``resampling'' step (see e.g. \cite{defreitas05,liu02}) in which copies of a system are weighted by the ratio of two densities and then resampled to produce an unweighted set of samples.  Some popular resampling algorithms are  adaptations of the generalised version of DMC that we will present below in Algorithm~\ref{dmc} (e.g. residual resampling as described in  \cite{liu02}).  Our results suggest that sequential importance sampling schemes could be improved (in some cases dramatically) by building the resampling step from our modification of DMC in Algorithm~\ref{tdmc}.




One can imagine a wide range of potential uses of DMC.  For example, we will show that the generalisation of DMC  in Algorithm~\ref{dmc} below could potentially be used to compute approximations to quantities of the form
\[
\mathbf{E}\left( f(y_t) \exp\Bigl({-\int_0^t V(y_t) dy_t}\Bigr)\right)\;,
\]
or
\begin{equ}[e:rareEvent]
\mathbf{E}\left( f(y_t) \exp\left({-V(y_t)}\right)\right)\;.
\end{equ}
where $y_t$ is a diffusion.  Continuous time data assimilation requires the approximation of expectations similar to first type above while, when applied to computing expectations of the type in \eqref{e:rareEvent}, DMC becomes a rare event simulation technique (see e.g. \cite{HammersleySIS:1954, RosenbluthSIS:1955, Allen:FFS:2006, Johansen:SMCrare:2006}), i.e.\ a tool for efficiently generating samples of extremely low probability events and computing their probabilities.   Such tools can be used  to compute the probability or frequency of dramatic fluctuations in a stochastic process.  Those fluctuations might, for example, characterise a reaction in a chemical system or a  failure event in an electronic device
 (see e.g. \cite{Frenkel1996,Bucklew:2004p3891}). 

The mathematical properties of DMC have been explored in great detail (see e.g. \cite{DelMoral:FK:2011}).  Particularly relevant to our work is the thesis \cite{Rousset:PhD:2006} which studies the continuous time limit of DMC  in the quantum Monte Carlo context  and introduces schemes for that setting that bear some resemblance to our modified scheme Algorithm~\ref{tdmc}.  The initial motivation for the present work is the observation that, in some of the most interesting settings, the natural generalisation of DMC introduced in Algorithm~\ref{dmc} exhibits a dramatic instability.  We show that
this instability can be corrected by a slight enlargement of the state space to 
include a parameter dictating which part of the state space
a sample is allowed to visit.  The resulting method is introduced in Algorithm~\ref{tdmc} in Section \ref{sec:summary} below. 
This modification is inspired by the RESTART and DPR rare event simulation algorithms \cite{Villen-Altamirano1991,HarasztiTownsend1999,MR2833612}.  The RESTART and DPR algorithms 
bias an underlying Markov process by splitting the trajectory of the process into multiple 
trajectories and then appropriately reweighing the resulting trajectories. 
The splitting occurs every time the process moves from one set to another in a
predefined partition of the state space.
The key feature of the algorithms that we borrow is a restriction placed on the 
trajectories resulting from a splitting event.
The new trajectories are only allowed to explore a certain region of state space 
and are eliminated when they exit this region. 
Our modification of DMC is similar, except that our branching rule is not 
necessarily based on a fixed partition of state space.

In Section~\ref{sec:ex}, we demonstrate the stability and utility of the new method in two applications.  In the first one, 
we use the method to compute the probability of very unlikely transitions in a small cluster of 
Lennard-Jones particles.  In the second example we show that the method can be used to efficiently 
assimilate high frequency observational data from a chaotic diffusion. 
In addition to these computational examples, we provide an in-depth mathematical study of the new scheme (Algorithm~\ref{tdmc}).  
We show that regardless of parameter 
regime  and even underlying dynamics (i.e.\ $y_t$ need not be related to a diffusion)
the estimator generated by the new algorithm has lower variance per workload than the 
DMC estimator.   
In a companion article \cite{brownianfan}, by focusing on a particular asymptotic regime (small time-discretisation parameter) 
we are also able to rigorously establish several dramatic results concerning the stability of 
Algorithm~\ref{tdmc} in settings in which the straightforward generalisation of DMC is unstable.  In particular, in that work we provide a characterization and proof of convergence in the continuous time limit to a well-behaved limiting Markov process which we call the Brownian fan.

\subsection{Notations}

In our description and discussion of the methods below we will use the standard notation $\lfloor a\rfloor = \max\{ i\in \mathbb{Z}:\,i\leq a\}$.


\subsection*{Acknowledgements}

{\small
We would like to thank H.~Weber for numerous discussions on this article and S.~Vollmer for
pointing out a number of inaccuracies found in earlier drafts.  We would also like to thank J.~Goodman and E.~Vanden-Eijnden for helpful conversations and advice during the early stages of this work.
Financial support for MH was kindly provided by EPSRC grant EP/D071593/1, by the 
Royal Society through a Wolfson Research Merit Award, and by the Leverhulme Trust through a Philip Leverhulme Prize.  JW was supported by NSF through award
DMS-1109731.
}

\section{The algorithm and a summary of our main results}\label{sec:summary}
With the potential term $V$ removed, the  PDE \eqref{psipde} is simply the familiar Fokker-Planck (or forward Kolmogorov) equation and 
 one can of course approximate \eqref{avet} (with $V=0$) by 
\[
 \widehat f _t  = \frac{1}{M}\sum_{j=1}^M f(w^{(j)}_t)\;,
 \]
 where the $w^{(j)}$ are $M$ independent realisations of a Brownian motion.  The
probabilistic interpretation of the source term $V\psi$ in the PDE is that it is responsible for the killing and creation of sample trajectories of $w$.  In practice one cannot exactly compute the integral appearing in the exponent in
\eqref{avet}.   Common practice (assuming that $t_k = k\eps$ ) is to replace the integral by the approximation
\[
\int_0^t V(w_s)\,ds \approx \sum_{k=0}^{\lfloor t/\eps\rfloor-1} \frac{1}{2}\left(V(w_{t_{k+1}})+V(w_{t_k})\right)\eps\;,
\]
where $\eps>0$ is a small time-discretisation parameter.
DMC then approximates
\begin{equation}\label{qmcave}
\langle f \rangle_t =  \mathbf{E} \biggl( f(w_t) \exp\Bigl(- \sum_{k=0}^{\lfloor t/\eps\rfloor-1} \frac{1}{2}\left(V(w_{t_{k+1}})+V(w_{t_k})\right)\eps \Bigr)\biggr).
\end{equation}
The version of DMC that we state below in Algorithm~\ref{dmc} is a slight generalisation of the standard algorithm and
approximates
  \begin{equation}\label{discavet2}
\langle f \rangle_t =  \mathbf{E}\biggl( f(y_t) \exp\Bigl(- \sum_{t_k\leq t}\v(y_{t_k},y_{t_{k+1}}) \Bigr)\biggr)\;,
\end{equation}
where $y_t$ is any Markov process and $\v$ is any function of two variables.  Strictly for convenience we will assume that the time $t$ is among the discretization points $t_0,t_1,t_2,\dots.$
The DMC algorithm proceeds as follows:
\begin{algorithm}{DMC}\label{dmc}
Slightly generalised DMC
{\tt
\begin{enumerate}
\item Begin with $M$ copies $x^{(j)}_0 = x_0$ and $k=0$.
\item At step $k$  there are $N_{t_k}$ samples $x^{(j)}_{t_k}$.  Evolve each of these $t_{k+1}-t_k$ units of time under the underlying dynamics  to generate
$N_{t_k}$ values \[\tilde x^{(j)}_{t_{k+1}}\sim \mathbf{P}\bigl( y_{t_{k+1}}\in dx\,|\, y_{t_k}=x^{(j)}_{t_k}\bigr).\]
\item For each $j=1,\dots, N_{t_k}$, let 
\[
P^{(j)} = e^{-\v(x^{(j)}_{t_k},\tilde x^{(j)}_{t_{k+1}})}
\]
and set
\[
N^{(j)} = \lfloor P^{(j)} + u^{(j)} \rfloor\;,
\]
where the $u^{(j)}$ are independent $\mathcal{U}(0,1)$ random variables.
\item For  $j=1,\dots, N_{t_k}$, and for $i=1,\dots,N^{(j)}$ set
\[
x^{(j,i)}_{t_{k+1}} = \tilde x^{(j)}_{t_{k+1}}.
\]
\item  Finally, set $N_{t_{k+1}} = \sum_{j=1}^{N_{k\eps}} N^{(j)}$ and list the $N_{t_{k+1}}$ vectors 
$\bigl\{x^{(j,i)}_{t_{k+1}}\bigr\}$  as $\bigl\{x^{(j)}_{t_{k+1}}\bigr\}_{j=1}^{N_{t_{k+1}}}$.
\item
At time $t,$ produce the estimate
\[
 \widehat f_t  = \frac{1}{M}\sum_{j=1}^{N_t} f(x^{(j)}_t)\;.
 \]
\end{enumerate}}
\end{algorithm}
Here, the notation $\CU(a,b)$ was used to refer to the law of a random variable that is uniformly distributed
in the interval $(a,b)$. 
Below we will refer to the $x^{(j)}_t$ as \emph{particles} and to a collection of particles as an \emph{ensemble}.

\begin{remark}\label{rem:cond}
If the goal is to compute the ``conditional'' quantity $\langle f\rangle_t / \langle 1\rangle_t$
then one can simply redefine $P^{(j)}$ in Step 3 by 
\[
P^{(j)} = \frac{e^{-\v(x^{(j)}_{t_k},\tilde x^{(j)}_{t_{k+1}})}}{\frac{1}{N_t}\sum_{l=1}^{N_t} e^{-\v(x^{(l)}_{t_k},\tilde x^{(l)}_{t_{k+1}})}}.
\]
With this choice the copy number $N_t$ becomes a Martingale and will need to be controlled at the cost of some ensemble size dependent bias.  The value of $N_t$ can be maintained exactly at some predefined value $M$ by uniformly upsampling or downsampling the ensemble after each iteration.  Alternatively one can hold $N_t$ near  $M$ by
multiplying this $P^{(j)}$ by $(M/N_t)^\alpha$ for some $\alpha>0$ as we do in Section  \ref{sec:filter}.
\end{remark}

Algorithm~\ref{dmc} results in an unbiased estimate of \eqref{discavet2}, i.e.
\[
\mathbf{E}^1 \widehat f_t =  \langle f\rangle_t\;.
\]
For reasonable choices of $\v$ and $f$ (e.g. $\sup_{t_k\leq t}\langle 1 \rangle_{t_k}<\infty$ and $\langle |f|\rangle_t<\infty$) the law of large numbers then implies that
\[
\lim_{M\rightarrow\infty} \widehat{f}_t =  \langle f\rangle_t.
\]
We use the symbol $\mathbf{E}^1$ for expectations under the rules of Algorithm~\ref{dmc} to distinguish them
for expectations under the rules of our modified algorithm (Algorithm~\ref{tdmc} below) which will simply be denoted by $\mathbf{E}$.

 Of course if we set
\begin{equation}\label{vqmc}
\v(x,y) = \frac{1}{2}\left(V(x)+V(y)\right)(t_{k+1}-t_k)\;,
\end{equation}
then we are back to the quantum Monte Carlo setting.  We might, however, choose
\begin{equation}\label{sint}
\v(x,y) = V(x)(y-x)\;,
\end{equation}
which, if $y_t$ approximates a diffusion (which we also denote by $y_t$),  formally corresponds to an approximation of 
\begin{equation*}
\langle f \rangle_t =  \mathbf{E} \left( f(y_t) \exp\Bigl(-\int_0^t V(y_s)dy_s\Bigr)\right)\;,
\end{equation*}
or 
\begin{equation}\label{reV}
\v(x,y) = V(y)-V(x)\;,
\end{equation}
corresponding to an approximation of
\begin{equation*}
\langle f \rangle_t =  e^{V(x_0)}\mathbf{E}\Bigl( f(y_t) e^{-V(y_t)}\Bigr).
\end{equation*}

The generalisation of DMC resulting from various choices of $\v$ is not frivolous.    In Section \ref{sec:ex}
we will see that, with $\v$ of form \eqref{reV}, one can design potentially very useful 
 rare event and particle filtering schemes. Unfortunately,  when  $y_t$ is a diffusion or a discretisation of a diffusion, the resulting Algorithms behave
extremely poorly in the small discretisation size limit ($t_{k+1}-t_k\rightarrow 0)$.
The reason is simple: on a time interval $[t_k,t_{k+1})$, a typical step of Brownian motion moves by
$\CO(\sqrt {t_{k+1}-t_k})$. Therefore, unlike for \eref{vqmc}, for both \eqref{sint} and \eqref{reV}, 
the typical size of $\v(x,y)$ is of order $\sqrt {t_{k+1}-t_k}$. As a consequence, at each step,
the probability that a particle either dies or spawns a child is itself of order $\sqrt {t_{k+1}-t_k}$. 
Without modification, as $t_{k+1}-t_k \to 0$, the fluctuations in the number of particles, $N_t$, will grow wildly and the process will die out before a time of order $1$ with higher and higher 
probability, thus causing the variance of our estimator to explode.  This phenomenon is already evident in the simple case of a Brownian motion,
\begin{equation}\label{cs1}
y_{(k+1)\eps} = y_{k\eps} + \sqrt{\eps} \xi_{k+1}\;,\qquad X(0) = 0\;,
\end{equation}
with 
\begin{equation}\label{cs2}
\v(x,y) = y-x
\end{equation}
(a choice consistent with either \eqref{sint} or \eqref{reV}). 
Here the $\xi_{k}$ are independent mean 0 and variance 1 Gaussian random variables.
This example (with more general $\xi_k$) is analysed in great depth in the companion paper \cite{brownianfan}.  

Figure \ref{nsquared} demonstrates the dramatic failure of the straightforward generalisation of DMC in Algorithm~\ref{dmc} on this simple example.  There we plot the logarithm of the second moment of the number of particles in the ensemble at time
$1$ (i.e. after $1/\eps$ steps of the algorithm) versus $-\log \eps$ for several values of $\eps$.  Clearly, $\mathbf{E}\left( N_{1}^2\right)$ is growing rapidly as $\eps$ decreases.  
Indeed, for small enough $\eps$ and  starting with a single initial copy at the origin, one can exploit the conditional independence of offspring in Algorithm~\ref{dmc} to show that for the simple random walk example, $\mathbf{E}\left[ N^2_{k\eps}\right] \geq \mathbf{E}\left[ N^2_{(k-1)\eps}\right]+\alpha \sqrt{\eps}$ for a positive constant $\alpha.$ After iterating $\eps^{-1}$ times, this bound gives exactly the growth of $\mathbf{E}\left[ N^2_{1}\right]$ illustrated in Figure \ref{nsquared} (i.e. $\mathcal{O}\left(\eps^{-1/2}\right)$).
This instability can be removed by carrying out the branching steps (Steps  3-5) only every $\mathcal{O}(1)$ units of time instead of every $\eps$ units of time and accumulating weights for the particles in the intervening time intervals.   However, this small $\eps$ regime cleanly highlights a serious (and unnecessary) deficiency in DMC.  In contrast to Algorithm~\ref{dmc}, the method that we will describe in Algorithm~\ref{tdmc} below, appears  to be stable as $\eps$ vanishes (a property confirmed in \cite{brownianfan}).

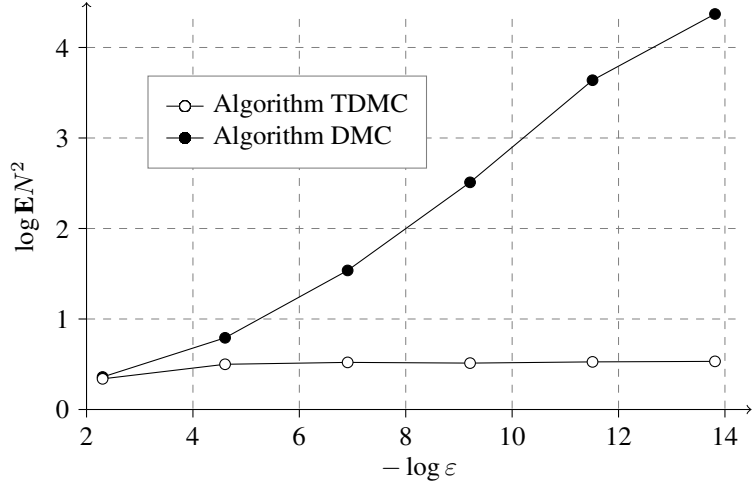
\begin{figure}\label{nsquared}
\begin{center}
\begin{tikzpicture}[y=1.2cm, x=.7cm]
	\draw[->] (2,0) -- coordinate (x axis mid) (14.5,0);
    	\draw[->] (2,0) -- coordinate (y axis mid) (2,4.5);
      	\foreach \x in {2,4,...,14}
     		\draw (\x,0pt) -- (\x,-3pt)
			node[anchor=north] {\x};
    	\foreach \y in {0,...,4}
     		\draw[xshift=1.4cm] (0pt,\y) -- (-3pt,\y) 
     			node[anchor=east] {\y}; 
	\node[below=0.5cm] at (x axis mid) {$-\log \eps$};
	\node[rotate=90, above=0.5cm] at (y axis mid) {$\log \E N^2$};

	\draw plot[mark=*] 
		file {DMC_Test.data};
	\draw plot[mark=*, mark options={fill=white} ] 
		file {TDMC_Test.data};
    
	
	\begin{scope}[shift={(3.5,3)}]
	\draw (0,0) node (dummy) {} (0,0) -- 
		plot[mark=*] (0.35,0) -- (0.7,0) 
		node[right] (alg1) {Algorithm~\ref{dmc}};
	\draw[yshift=\baselineskip] (0,0) -- 
		plot[mark=*, mark options={fill=white}] (0.35,0) -- (0.7,0)
		node[right] (alg2) {Algorithm~\ref{tdmc}};
	\end{scope}

	\begin{pgfonlayer}{background}
  	\draw[very thin,color=gray,dashed,xstep=2,ystep=1] (2.01,0.01) grid (14.2,4.4);
    	\node [fill=white,draw=gray,fit=(alg1) (alg2) (dummy)] {};
  	\end{pgfonlayer}

\end{tikzpicture}
\end{center}
\vspace{-1em}\caption{Second moment of $N$ versus stepsize.}
\end{figure}

Our modification of DMC described in Algorithm~\ref{tdmc} below not only suppresses  the small $\eps$ instability just described, but results in a more robust and efficient algorithm in any context.
Inspired by the RESTART and DPR algorithms studied in \cite{Villen-Altamirano1991,HarasztiTownsend1999,MR2833612} we
append to each particle a ``ticket'' $\theta$ limiting the region that the particle is allowed to visit.  Samples are eliminated when and only when they violate their ticket values. As a consequence, particles typically survive for a much longer time, and in many
situations there is a well-defined limiting algorithm when $\eps \to 0$.
More precisely, the modified algorithm proceeds as follows:
\begin{algorithm}{TDMC}\label{tdmc}
Ticketed DMC
{\tt
\begin{enumerate}
\item Begin with $M$ copies $x^{(j)}_0 = x_0$.  For each $j=1,\dots,M$ choose an independent random variable
$\theta^{(j)}_0\sim \mathcal{U}(0,1)$.
\item At step $k$  there are $N_{t_k}$ samples $(x^{(j)}_{t_k},\theta^{(j)}_{t_k})$.  Evolve each of the $x^{(j)}_{t_k}$ one step to generate
$N_{t_k}$ values  \[\tilde x^{(j)}_{t_k}\sim \mathbf{P}\bigl( y_{t_{k+1}}\in dx\,|\, y_{t_k}=x^{(j)}_{t_k}\bigr).\]
\item For each $j=1,\dots, N_{t_k}$, let
\begin{equ}[e:defP]
P^{(j)} = e^{-\v(x^{(j)}_{t_k},\tilde x^{(j)}_{t_{k+1}})}.
\end{equ}
If $P^{(j)} < \theta^{(j)}_{t_k}$ then
set 
\[
N^{(j)} = 0.
\]
If $P^{(j)} \geq \theta^{(j)}_{t_k}$ then set
\begin{equ}[e:noff]
N^{(j)} = \max\{\lfloor P^{(j)} + u^{(j)} \rfloor , 1\}\;,
\end{equ}
where $u^{(j)}$ are independent $\mathcal{U}(0,1)$ random variables.
\item For  $j=1,\dots, N_{t_k}$, if $N^{(j)}>0$ set 
\[
x^{(j,1)}_{t_{k+1}} = \tilde x^{(j)}_{t_{k+1}}\quad\text{and}\quad \theta^{(j,1)}_{t_{k+1}} = \frac{\theta^{(j)}_{t_k}}{P^{(j)}}
\]
and for $i=2,\dots,N^{(j)}$ 
\[
x^{(j,i)}_{t_{k+1}} = \tilde x^{(j)}_{t_{k+1}}\quad\text{and}\quad  \theta^{(j,i)}_{t_{k+1}}\sim  \mathcal{U}((P^{(j)})^{-1},1).  
\]
\item  Finally set $N_{t_{k+1}} = \sum_{j=1}^{N_{t_k}} N^{(j)}$ and list the $N_{t_{k+1}}$ vectors 
$\bigl\{x^{(j,i)}_{t_{k+1}}\bigr\}$  as $\bigl\{x^{(j)}_{t_{k+1}}\bigr\}_{j=1}^{N_{t_{k+1}}}$.
\item
At time $t$ produce the estimate
\[
 \widehat  f _t  = \frac{1}{M}\sum_{j=1}^{N_t} f(x^{(j)}_t)\;.
 \]
\end{enumerate}
}\end{algorithm}

\begin{remark}
Both here and in Algorithm~\ref{dmc}, we could have replaced the random variable $\lfloor P^{(j)} + u^{(j)} \rfloor$ by 
any integer-valued random variable with mean $P^{(j)}$. The choice given here is natural, since it
is the one that minimises the variance. From a mathematical perspective, the analysis would have been slightly easier
if we chose instead to use Poisson distributed random variables.
\end{remark}

We should first establish that this new algorithm ``does the job'' in the sense that if  the steps in 
Algorithm~\ref{dmc} are replaced
by those in Algorithm~\ref{tdmc}, the resulting ensemble of particles still produces an unbiased estimate of the quantity $\langle f\rangle_t$ 
defined in \eqref{discavet2}.  This is the subject of Theorem~\ref{theorem:mean} below 
which establishes that, indeed, for any bounded test function $f$ and any choice of $\v,$ 
\begin{equation}\label{mean2}
\mathbf{E} \widehat  f _{t} =  \langle f\rangle_{t}\;.
\end{equation}

We have already mentioned that a companion article \cite{brownianfan} is devoted to showing that our modification of DMC, Algorithm~\ref{tdmc}, does not suffer the small $\eps$ instability observed in Algorithm~\ref{dmc}.  In that asymptotic regime Algorithm~\ref{tdmc} dramatically outperforms the straightforward generalisation of DMC in Algorithm~\ref{dmc}.  However, a surprising side-effect of our modification is that Algorithm~\ref{tdmc} is superior to Algorithm~\ref{dmc} in all contexts. In this article we compare Algorithms~\ref{dmc} and \ref{tdmc} independently of a 
specific regime (discretisation stepsize, choice of $\v$, etc).  

The key tool in carrying out this comparison is Lemma~\ref{tdmc2dmc} in Section~\ref{sec:bias+dmc}.  That lemma establishes that by simply randomising all tickets uniformly between $0$ and 1 at each step of Algorithm~\ref{tdmc}, one obtains a process that, in law, is identical to the one generated by Algorithm~\ref{dmc}.
With this observation in hand we are able to establish in Theorem~\ref{theorem:varcompare} that
 the estimator produced by Algorithm~\ref{tdmc} always has lower variance than Algorithm~\ref{dmc}, i.e.\ we prove that
\begin{equation}\label{varcompare}
\mathop{\mathbf{var}} \widehat  f _t \leq \mathop{\mathbf{var}^1}  \widehat  f _t \;,
\end{equation}
holds for any bounded test function $f,$ any underlying Markov chain $Y$, and any choice of $\v$.
In expression \eqref{varcompare} we have again distinguished expectations under the rules of Algorithm~\ref{dmc} from all other expectations by a superscript 1.

In comparing Algorithms~\ref{dmc} and \ref{tdmc}, it is not enough to simply compare  the variances of the corresponding estimators: one should also compare their respective costs.  Since the dominant contribution to any cost difference in the two algorithms will come from differences in the number of times one must evolve a particle from one time step to the next we compare the expectations of \emph{workload} 
\[
\mathcal{W}_t = \sum_{t_k\leq t} N_{t_k}\;
\]
under the two rules.  However, by \eqref{mean2} with $f\equiv 1$ we see that
\[
\mathbf{E}^1 \mathcal{W}_t = \mathbf{E} \mathcal{W}_t =  M \sum_{t_k\leq t} \mathbf{E} \exp\Bigl(- \sum_{j=0}^{k-1}\v(y_{t_j},y_{t_{j+1}}) \Bigr)\;,
\]
so that the expected cost of both algorithms is the same.
In fact, the proof of Theorem~\ref{theorem:varcompare} can be modified slightly to show that the variance of the workload of Algorithm~\ref{tdmc} is always lower than the variance of workload of Algorithm~\ref{dmc}.
These remarks leave little room for ambiguity.  Algorithm~\ref{tdmc} is more efficient than Algorithm~\ref{dmc} in any setting.

The existence of a continuous-time limit and other robust, small-discretisation parameter characteristics established  in \cite{brownianfan} along with the other features of Algorithm~\ref{tdmc} mentioned above, strongly suggest that Algorithm~\ref{tdmc} has significant potential as a flexible and efficient computational tool.  In the next section we provide numerical evidence to that effect.

\section{Two examples}\label{sec:ex}
In the following subsections we apply our modified DMC algorithm to two simple problems.  In Section \ref{sec:re} we consider the approximation of a small probability related to the escape of a diffusion from a deep potential well.  We will use Algorithm~\ref{tdmc}
 to bias simulation of the diffusion so that an otherwise rare event (escape from the potential well) becomes common (and then reweight appropriately).  We will see that extremely low probability events can be computed with reasonably low relative error.  In Section \ref{sec:filter} we demonstrate the performance of Algorithm~\ref{tdmc} in one of DMCs most common application contexts, particle filtering.   We will reconstruct a hidden trajectory of a diffusion, given noisy but very frequent observations.  Our comparison shows that the reconstruction produced by the modified method is dramatically more accurate than the reconstruction produced by straightforward generalisation of DMC.

\subsection{Rare event simulation}
\label{sec:re}

In  Quantum Monte Carlo applications, the function $\v$ is (for the most part) specified by the potential $V$ as in \eqref{vqmc}.  In other applications however, there may 
be more freedom in the choice of $\v$
to achieve some specific goal.  As already mentioned earlier, the choice
$\v(x,y) = V(y) - V(x)$ turns Algorithm~\ref{tdmc} into an estimator of
\[
\langle f\rangle_t = 
e^{V(x_0)}\mathbf{E}\left(
f(y_t)e^{-V(y_t)}\right)\;.
\]
The design of a ``good'' choice of $V$ will be discussed below.
By redefining $f$ we find that $e^{-V(x_0)}\, \widehat {e^{V}f}_t$ is an unbiased estimate of $\mathbf{E}\left(f(y_t)\right)$,
i.e.
\[
e^{-V(x_0)}\,\mathbf{E} \sum_{j=1}^{N_t} e^{V(x^{(j)}_t)} f(x^{(j)}_t)  = \mathbf{E} f(y_t)\;,
\]
where the samples $x^{(j)}_t$ are generated as in Algorithm~\ref{tdmc}.  This suggests choosing $V$ to minimise the variance of
$ \widehat {e^{V}f}_t$.  Intuitively this involves choosing  $V$ to be smaller in regions where $f$ is significant.
The branching step in Algorithm~\ref{tdmc} will create more copies when  $V$ decreases, focusing computational resources in regions where $V$ is small. 

As an example, consider the overdamped Langevin dynamics 
\[
 dy_t = -\nabla U(y_t)dt + \sqrt{2\gamma}dB_t
\]
where $y$ represents the positions of seven two-dimensional particles (i.e. $y\in \mathbb{R}^{14}$)
and
\[
 U(x) = \sum_{i<j} 4\left(\lVert x_i -x_j\rVert^{-12}- \lVert x_i
-x_j\rVert^{-6}\right)\;.
\]
In this formula $x_i$ represents the position of the $i$th particle (i.e. $x_i\in \mathbb{R}^2$).
In various forms this Lennard-Jones cluster is a standard non-trivial test problem in rare event simulation (see \cite{DellagoBC1998}).

\begin{figure}
\centering
\begin{tikzpicture}[minimum size=8mm]
\foreach \x in {0,...,5} {
  \node at (60*\x:1cm) [circle,draw=black!80,fill=blue!10] {};
 }
  \node at (0,0) [circle,draw=black!80,fill=blue!40] {};

\def\mynode{\node[xshift=4cm+0.1cm*rand,yshift=0.1cm*rand]}

\foreach \x in {1,...,5} {
  \mynode at (60*\x:1cm) [circle,draw=black!80,fill=blue!10] {};
 }
  \mynode at (0,0) [circle,draw=black!80,fill=blue!10] {};
  \mynode at (1,0) [circle,draw=black!80,fill=blue!40] {};
 \end{tikzpicture}
\caption{({\bf Left}) The initial condition $x^A$ to generate the results in Table \ref{results}. ({\bf Right}) A typical sample of the event $y_{2}\in B$.}
\label{LJfig}
\end{figure}
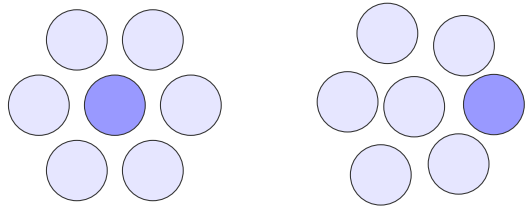
After discretising this process by the Euler scheme 
\[
y_{(k+1)\eps} =y_{k\eps} - \nabla U(y_{k\eps})\, \eps + \sqrt{2\gamma \eps}\, \xi_{k+1}
\]
 with a stepsize of $\eps = 10^{-4}$ we apply Algorithm
\ref{tdmc} with 
\[
 \v(x,y) = V(y)-V(x),\qquad V(x) =  \frac{\lambda}{kT} \min_{i\geq 2}\bigg\{ \bigg\lVert x_i - \frac{1}{7}\sum_j x_j\bigg\rVert\bigg\}.
\]
The properties of the Euler discretisation in the context of rare event simulation are discussed in \cite{EveWeare2012}.

  Our goal is to compute
\[
\mathbf{P}_{x^A}\left(y_{2}\in B\right),\qquad B = \left\{x:\, V(x)<0.1\right\}\;,
\]
where our initial configuration is given by $x^A_1 = (0,0)$ and for $j=2,3,\dots,7$, 
\[
x^A_j = \Bigl(\cos\frac{j\pi}{3},\sin\frac{j\pi}{3} \Bigr)
\]
(see Figure \ref{LJfig}).  Starting from initial configuration $x_0=x^A$, the particle initially at the centre of the cluster ($x^A_1$) will typically remain there for a length of time that increases exponentially in ${\gamma}^{-1}$. We are roughly  computing the probability that, in 2 units of time, the particle at the centre of the cluster exchanges positions with one of the particles in the outer shell.  This probability decreases exponentially in ${\gamma}^{-1}$ as
$\gamma\rightarrow 0$.
In this case 
\[
f(x) = \one_B(x)\;,
\]
where $\one_B$ is the indicator function of $B$.  In our simulations,
$\lambda$ is chosen so that the expected number of particles ending in $B$, $\mathbf{E}\widehat f_2$, is
(very) roughly 1.  The results for several values of the temperature $\gamma$ are displayed in Table \ref{results}.

 \begin{table}
\centering
\begin{tabular*}{\hsize}{@{\extracolsep{\fill}}l l l l l l}
\quad $\gamma$ & $\lambda$ & estimate & workload &
$\frac{1}{2}\left(\substack{\text{variance}\\ \times\text{workload}}\right)$ &
$\substack{\text{brute force}\\ \text{variance}}$
\quad \\[1ex]
\hline
\quad 0.4  & 1.9 & 1.125$\times 10^{-2}$   & 12.90  &  4.732$\times 10^{-3}$  & 1.112$\times 10^{-2}$ \quad\\
\quad 0.2 & 1.3 & 2.340$\times 10^{-3}$ & 11.64 &2.344$\times 10^{-4}$ & 2.335$\times
10^{-3}$ \quad\\
\quad 0.1 & 1 & 7.723$\times 10^{-5}$& 13.87 & 7.473$\times 10^{-7}$ & 7.722$\times
10^{-5}$ \quad \\
\quad 0.05 & 0.85 & 9.290$\times 10^{-8}$ & 30.84 & 1.002$\times
10^{-11}$   &   9.290$\times 10^{-8}$   \quad \\
\quad 0.025 & 0.8 & 1.129$\times 10^{-13}$& 204.8 & 1.311$\times
10^{-21}$   &   1.129$\times 10^{-13}$   \quad \\
\hline
\end{tabular*}
\caption{Performance of Algorithm~\ref{tdmc} on the Lennard-Jones cluster problem described in 
Section \ref{sec:re} at different temperatures ($\gamma$).  A measure of the efficiency improvement of Algorithm~\ref{tdmc} over straightforward simulation is obtained by taking the ratio of the last two columns}\label{results}
\end{table}


The workload referenced in Table \ref{results} is the (scaled) expected total number of $\nabla U$ evaluations per sample, i.e. the expectation of 
\[
\eps \mathcal{W}_2 = {\eps}\sum_{k=1}^{2/\eps} N_{k\eps}.
\]
Note that when $V\equiv 0$, Algorithm~\ref{tdmc} reduces to straightforward (brute force) simulation and the workload is exactly 2.  
  Increasing $\lambda$ results in the creation of more particles.  This has the competing effects of increasing the workload on the one hand but reducing the variance of the resulting estimator on the other.  If we let $\sigma^2$ denote the variance of one sample run of Algorithm~\ref{tdmc} (i.e.  $\mathbf{var}\, e^{-V(x^A)}\widehat {e^{V}f}_{2}$ with $M=1$)
 and $\sigma^2_{bf}$ the variance of the random variable $f(y_{2})$, then the number of samples required to compute an estimate with a statistical error of $err$ is $M = \sigma^2/err$ and $M_{bf} = \sigma^2_{bf}/err$ for
 Algorithm~\ref{tdmc} and brute force respectively.  Taking account of the expected cost (scaled by $\eps$) per sample of Algorithm~\ref{tdmc}
(reported in the ``workload'' column) 
and of brute force simulation (identically 2) we can obtain a comparison of the cost of Algorithm~\ref{tdmc} and brute force by comparing the brute force variance to the product of the variance of the estimate produced by Algorithm~\ref{tdmc} and one-half the workload reported in the fourth column of the table.   These quantities are reported in the last two columns of the table.  By taking the ratio of the two columns we obtain a measure of the relative cost per degree of accuracy of the two methods.
    One can see that  the speedup provided by Algorithm~\ref{tdmc} becomes dramatic for smaller values of $\gamma$.  The  variance of the brute force  estimator is computed from the values in the third column  by
  \[
  \text{brute force variance} = \mathbf{P}_{x^A}\left(y_{2}\in B\right)\left( 1- \mathbf{P}_{x^A}\left(y_{2}\in B\right)\right).
  \]
    Only the estimates in the first two rows of Table \ref{results} were compared against straightforward simulation.  Those estimates agreed to the appropriate precision.

The choice of $V$ used in the above simulations is certainly far from optimal.   Given the similarities in this rare event context between Algorithm~\ref{tdmc} and the DPR algorithm, it should be straightforward to adapt the 
results of Dean and Dupuis in \cite{MR2833612} identifying optimal choices of $V$ in the small $\gamma$ limit.  In fact, we suspect (but do not prove) that, at least when the step-size parameter is small, the optimal choice of $V$ at any temperature is given by the time-dependent function
\[
V(k,x) = - \log \mathbf{P}\left( y_{2} \in B \, |\, y_{k\eps} = x\right).
\]
This choice is clearly impractical as it requires knowing the probability that we are trying to compute.  Even asymptotic estimates based on taking an appropriate low $\gamma$ limit of this choice are often not practical so it is worth noting that, even with a rather clumsy choice of $V$, our scheme is still able to produce impressive results to fairly low temperatures.  We fully expect however that without a very carefully chosen $V$, the cost of achieving a fixed relative accuracy for our scheme will grow exponentially with $\gamma^{-1}$ as $\gamma\rightarrow 0$.   This characteristic is common in rare event simulation methods.

\subsection{Non-linear filtering}\label{sec:filter}

In this section our goal is to reconstruct a  sample path of the solution to the stochastic differential equation
\begin{align}\label{lorenz}
dy_t^{(1)} &= 10(y_t^{(1)}-y_t^{(2)})\,dt   + \sqrt{2} dB_t^{(1)}\;,\\
dy_t^{(2)} &= \bigl(y_t^{(1)}(28-y_t^{(3)})- y_t^{(2)}\bigr) \,dt  + \sqrt{2} dB_t^{(2)}\;, \\
dy_t^{(3)} & = \Bigl(y_t^{(1)}y_t^{(2)} - \frac{8}{3} y_t^{(3)}\Bigr)\,dt + \sqrt{2} dB_t^{(3)}\;, 
\end{align}
with
\[
y_0 = \left(\begin{array}{c}
-5.91652\\
-5.52332\\
24.57231
\end{array}
\right)\;.
\]
The deterministic version of this system is the famous Lorenz 63 chaotic ordinary differential equation \cite{Lorenz}.  The stochastic version above is commonly used in tests of on-line filtering strategies (see e.g. \cite{MR1271676,chorin}).

The path of $y_t$ is revealed only through the 3-dimensional noisy observation process
\begin{equation}\label{lorenzobs}
dh_t = 
y_t\, dt
+ 0.1\, d\tilde B_t\;
\end{equation}
where the 3-dimensional Brownian motion $\tilde B_t$ is independent of $B_t$.
Our task is to reconstruct the path of $y_t$ given a realisation of $h_t$.  More precisely, our goal is to sample from (and compute moments of) the conditional distribution of $y_t$ given $\mathcal{F}^h_t$, where
 $\mathcal{F}^h$ is the filtration generated by $h$.  One can verify that expectations with respect to this conditional distribution can be written as
 \begin{equation}\label{exactfilter}
 \frac{\mathbf{E}^B \left( f(y_t) \exp\Bigl( - 10 \int_0^t \langle y_s, dh(s)\rangle- 50 \int_0^t \lVert y_s\rVert^2 ds \Bigr) \right)   }
 {\mathbf{E}^B \exp\Bigl( - 10 \int_0^t \langle y_s, dh_s\rangle- 50 \int_0^t \lVert y_s\rVert^2 ds \Bigr) }\;,
 \end{equation}
 where the superscript $B$ on the expectations indicates that they are expectations over $B_t$ only and not over $\tilde B_t$, i.e.\ the trajectory of $h_t$ is fixed.
 We will focus on estimating the mean of this distribution ($f(x) = x$) which we will refer to as the reconstruction of the hidden signal. 

As always, we should first discretise the problem.  The simplest choice is to again replace \eqref{lorenz}
by its Euler discretisation with parameter $\eps$.  
The observation process \eqref{lorenzobs} can be replaced by
\begin{equation}\label{lorenzobsdisc}
  h_{(k+1)\eps} - h_{k\eps} \approx
y_{k\eps}\,\eps
+ 0.1\sqrt{\eps}\, \eta_{k+1}
\end{equation}
where the $\eta_k$ are independent Gaussian random variables with mean $0$ and identity covariance.  We again set $\eps=10^{-4}$. With this choice for $h_t$, it is equivalent to assume that the observation process is the sequence of increments $ h_{(k+1)\eps} -  h_{k\eps}$  instead of $H$ itself. 
To emphasise that we will be conditioning on the values of $h_{(k+1)\eps} -  h_{k\eps}$ and not computing expectations with respect to these variables we will use the notation 
$\triangle_k$ to denote a specific realisation of  $ h_{(k+1)\eps} -  h_{k\eps}$. In Figure \ref{fig:lorenz} we plot a trajectory of the resulting discrete process $y$ and observations $h_{(k+1)\eps} -  h_{k\eps}$.  
\begin{figure}
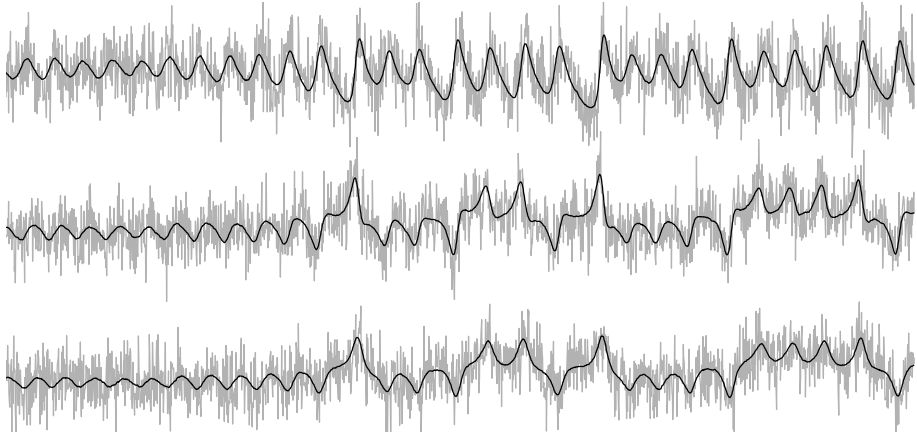
\label{fig:lorenz}
\centering
\mhpastefig{Observation}
\caption{Hidden signal (black), together with the observed signal in gray.}
\end{figure}

Note that given a particular state $y_{k\eps}=x$, the probability of observing $h_{(k+1)\eps} -  h_{k\eps}=\triangle$ is
\[
\exp\left( - \frac{\lVert x\eps- \triangle\rVert^2}{0.02\eps}\right)
\]
and expectations of the discretised process $y$ given the $h_{(k+1)\eps} -  h_{k\eps} = \triangle_k$ observations can be computed by the formula
\begin{equation}\label{discfilter}
Z_k^{-1} \mathbf{E}\biggl( y_{k\eps} \exp\Bigl(- \sum_{j=1}^{k}  \frac{\lVert y_{j\eps}\,\eps- \triangle_j\rVert^2}{0.02\eps}\Bigr) \biggr)\;,
\end{equation}
where the normalisation constant $Z_k$ is given by
\[
Z_k = {\mathbf{E}  \exp\Bigl(- \sum_{j=1}^{k}  \frac{\lVert y_{j\eps}\,\eps- \triangle_j\rVert^2}{0.02\eps}\Bigr) }\;.
\]
In the small $\eps$ limit, formula \eqref{discfilter} (with $k = t/\eps$)  indeed converges to \eqref{exactfilter}  (see \cite{Crisan2011}). 

Expression \eqref{discfilter} suggests applying Algorithms \ref{dmc} or \ref{tdmc} with 
\[
\v(x,y) = \frac{\lVert y\eps-\triangle_{k+1}\rVert^2}{0.02\eps}
\]
at each time step.  Below it will be convenient to consider the resulting choice of $P^{(i)}$,
\[
P^{(i)} = \exp\left( - \frac{\lVert \tilde x^{(i)}_{(k+1)\eps}\eps- \triangle_{k+1}\rVert^2}{0.02\eps}\right)
\]
in Algorithm~\ref{tdmc}.  With this choice, the expected number of particles at the $k$th step would be
\[
M {\mathbf{E}\exp\Bigl(- \sum_{j=1}^{k}  \frac{\lVert y_{j\eps}\,\eps- \triangle_j\rVert^2}{0.02\eps}\Bigr) }\;,
\]
a quantity that my decay very rapidly with $k$.  Since our goal is to compute conditional expectations we are free to normalise $P^{(i)}$ by
any constant.  A tempting choice is 
\[
P^{(i)} =\frac{Z_k}{Z_{k+1}}\exp\left( - \frac{\lVert \tilde x^{(i)}_{(k+1)\eps}\eps- \triangle_{k+1}\rVert^2}{0.02\eps}\right)
\]
which would result in $\mathbf{E} N_{k\eps} = M$ for all $k$.
Unfortunately we do not know the conditional expectation in the denominator of this expression.   However, for a large number of particles $N_{k\eps}$ we have the approximation
\begin{equs}
\frac{1}{N_{k\eps}}\sum_{i=1}^{N_{k\eps}}  \exp&\left( - \frac{\lVert \tilde x^{(i)}_{(k+1)\eps}\eps-\triangle_{k+1}\rVert^2}{0.02\eps}\right)\\
&\approx \mathbf{E}\Bigl( \exp \Bigl( - \frac{\lVert y_{(k+1)\eps}\eps-\triangle_{k+1}\rVert^2}{0.02\eps}\Bigr) \,\Big|\, \triangle_1,\dots,\triangle_k\Bigr) = \frac{Z_{k+1}}{Z_k}.
\end{equs}
  This suggests using
\begin{equation}\label{Pfirst}
P^{(i)} = \frac{\exp\left( - \frac{\lVert \tilde x^{(i)}_{(k+1)\eps}\eps- \triangle_{k+1}\rVert^2}{0.02\eps}\right)}{\frac{1}{N_{k\eps}}\sum_{l=1}^{N_{k\eps}}  \exp\left( - \frac{\lVert \tilde x^{(l)}_{(k+1)\eps}\eps- \triangle_{k+1}\rVert^2}{0.02\eps}\right)}\;,
\end{equation}
which will guaranty that $\mathbf{E} N_{k\eps} = M$ for all $k$.

In fact, in practical applications it is important to have even more control over the population size.  
Stronger control on the number of particles can be achieved in many ways (e.g. by resampling strategies \cite{defreitas05, liu02}).  We choose to make the simple modification
\begin{equation}\label{Psecond}
P^{(i)} = M\frac{\exp\left( - \frac{\lVert \tilde x^{(i)}_{(k+1)\eps}\eps-\triangle_{k+1}\rVert^2}{0.02\eps}\right)}{\sum_{l=1}^{N_{k\eps}}  \exp\left( - \frac{\lVert \tilde x^{(l)}_{(k+1)\eps }\eps- \triangle_{k+1}\rVert^2}{0.02\eps}\right)}\;,
\end{equation}
which results in an expected number of particles at step $k+1$ of $M$ independently of the details of the step $k$ ensemble.  Formula \eqref{Psecond} depends on the entire ensemble of walkers and does not fit strictly within the confines of Algorithms \ref{dmc} and \ref{tdmc}.  This choice will lead to estimators with a small, $M$-dependent, bias.  
All sequential Monte Carlo strategies with ensemble population control  for computing conditional expectations that we are aware of suffer a similar bias.

\begin{figure}
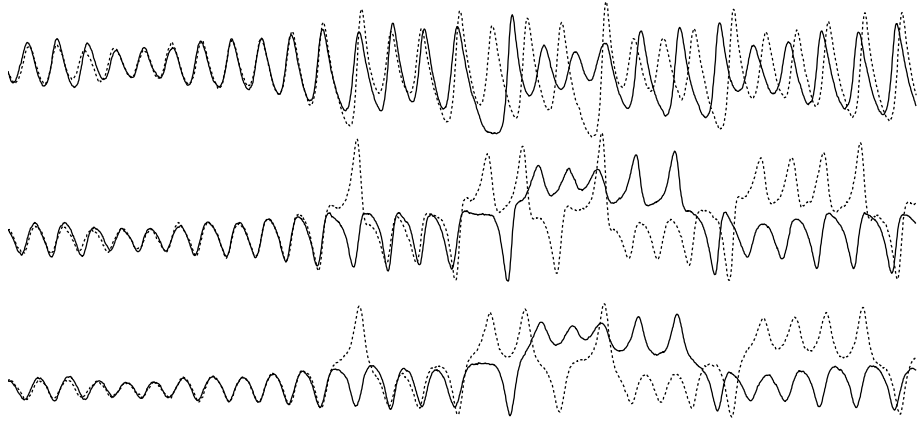

\centering
\mhpastefig{BadFilter}
\caption{Reconstruction of the three components of the signal using Algorithm~\ref{dmc}. The hidden signal is shown as a dotted line.}\label{fig:dmcfilter}
\end{figure}

\begin{figure}
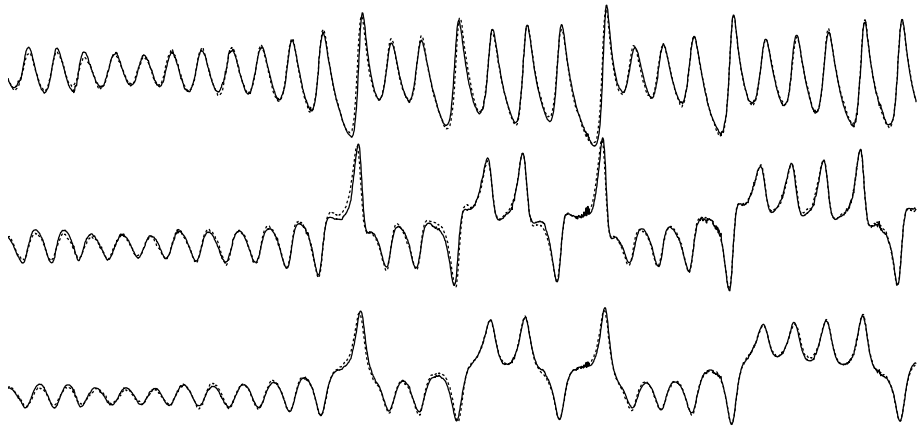

\centering
\mhpastefig{GoodFilter}
\caption{Reconstruction of the three components of the signal using Algorithm~\ref{tdmc}. The hidden signal is shown as a dotted line.}\label{fig:tdmcfilter}
\end{figure}

The results of a test of Algorithm~\ref{tdmc} with this choice of $P^{(i)}$ are presented in Figure \ref{fig:tdmcfilter} for $M=10$.  The true trajectory of $y$ is drawn as a dotted line, while our reconstruction is drawn as a solid line.  Note that the dotted line is nearly completely hidden behind the solid line, indicating an accurate reconstruction.  In Figure \ref{fig:dmcfilter} we show the results of the same test with Algorithm~\ref{tdmc} replaced by Algorithm~\ref{dmc}.  The method obtain in this way from Algorithm~\ref{dmc} is very similar to a standard particle filter with the common residual resampling step (see \cite{liu02}).  

With our small choice of $\eps$, one might expect that the number of particles generated by Algorithm~\ref{dmc} would explode or vanish almost instantly.  Our choice of $P^{(i)}$ prevents large fluctuations in $N$.  We expect instead that  the number of truly distinguishable particles in the ensemble (in the sense that they are not very nearly in the same positions) will drop to $1$ instantly, leading to a very low resolution scheme and a poor reconstruction.  This is supported by the results shown in Figure \ref{fig:dmcfilter}.

The fact that we obtain such a good reconstruction in Figure \ref{fig:tdmcfilter} with only 10 particles 
indicates that this is not a particularly challenging filtering problem.  Challenging filtering problems and more involved methods for dealing with them are discussed in \cite{EveWeare2012b}.  The purpose of this example is to demonstrate that by simply removing unnecessary variance from a particle filter (via Algorithm~\ref{tdmc}) one can improve performance.  
This improvement is illustrated dramatically in our small $\eps$ setting.  In practice it is advisable to carry out the copying and killing steps in Algorithm~\ref{tdmc} less frequently, accumulating particle weights in the intervening time intervals.  Nonetheless, any variance in the resulting estimate will be unnecessarily amplified if one uses a method based on Algorithm~\ref{dmc} instead of Algorithm~\ref{tdmc}.

\section{Bias and comparison of Algorithms \ref{dmc} and \ref{tdmc}}\label{sec:bias+dmc}

In this section we establish several general properties of Algorithms \ref{dmc} and \ref{tdmc}.   In contrast to later sections we do not focus here on any 
specific choice of the Markov process $y$ or the function $\v$.   
For simplicity in our displays we will assume in this section that $y$ is time-homogenous, i.e. that its transition probability distribution is independent of time.   As above  we use $\mathbf{E}^1$ and $\mathbf{E}$ to denote expectations under the rules of  Algorithms \ref{dmc} and \ref{tdmc} respectively.  In the proof of each result we will assume that $M=1$ in Algorithms \ref{dmc} and \ref{tdmc}.  The results follow for $M>1$ by the independence of the processes generated from each copy of the initial configuration.  To keep the formulas slightly more compact and because the results in these section do not pertain to the continuous time limit, we will replace our subscript $t_k$ notation with a subscript $k$ (and $t$ is now a non-negative integer).

Our first main result in this section establishes that Algorithms \ref{dmc} and \ref{tdmc} are unbiased in the following sense.
\begin{theorem}\label{theorem:mean}
Algorithm~\ref{tdmc} produces an unbiased estimator, namely
\begin{equation}\label{mean}
\mathbf{E} \widehat  f _t  = \mathbf{E}\biggl( f(y_t)\, \exp\Bigl({-\sum_{k=0}^{t-1} \v(y_{k},y_{{k+1}})}\Bigr) \biggr)\;.
\end{equation}
The  expression also holds with $\mathbf{E}$ replaced by $\mathbf{E}^1$.
\end{theorem}
The proof of Theorem~\ref{theorem:mean} is similar to the proof of Theorem~3.3 in  \cite{MR2833612} and requires the following lemma.
\begin{lemma}\label{lemma:onestep}
 For any bounded, non-negative function  $F$ on $\mathbb{R}^d\times \mathbb{R}$ we have 
\begin{equation}\label{onestep1}
\mathbf{E}\Bigl( \sum_{j=1}^{N_1} F(x^{(j)}_1, \theta^{(j)}_1)\,\Big \vert\, x_0, \tilde  x^{(1)}_1 \Bigr) =  e^{-\v(x_0,\tilde x^{(1)}_1)} \int_0^{1} F(\tilde x^{(1)}_1, u)\,du\;.
\end{equation}
The  expression also holds with $\mathbf{E}$ replaced by $\mathbf{E}^1$.
\end{lemma}


\begin{proof}
From the description of the algorithm and, in particular the fact that given $x_0$, and $\tilde x^{(1)}_1$, the tickets are all independent of each other,
 we obtain
\begin{equs}
\mathbf{E}\biggl( \sum_{j=1}^{N_1} F(x^{(j)}_1, \theta^{(j)}_1)\,\bigg \vert\, x_0&, \tilde x^{(1)}_1\biggr) = 
\int_0^{1}F(\tilde x^{(1)}_1, u e^{\v(x_0,\tilde x^{(1)}_1)} )\,\mathbf{1}_{\bigl(e^{-\v(x_0,\tilde x^{(1)}_1)} \geq u\bigr)}\,du \\
& + 
\frac{e^{-\v(x_0,\tilde x^{(1)}_1)}-1}{1 - e^{\v(x_0,\tilde x^{(1)}_1)}}
\int_{e^{\v(x_0,\tilde x^{(1)}_1)}}^1F(\tilde x^{(1)}_1, u)\,du\, \mathbf{1}_{\left(\v(x_0,\tilde x^{(1)}_1) \leq 0\right)}\;.
\end{equs}
The first term on the right can be rewritten as
\begin{multline*}
e^{-\v(x_0,\tilde x^{(1)}_1)} \int_0^{e^{\v(x_0,\tilde x^{(1)}_1)}}F(\tilde x^{(1)}_1, u)\,du\,\mathbf{1}_{\left(\v(x_0,\tilde x^{(1)}_1)\leq 0\right)} \\
+ e^{-\v(x_0,\tilde x^{(1)}_1)} \int_0^{1}F(\tilde x^{(1)}_1, u)\,du\,\mathbf{1}_{\left( \v(x_0,\tilde x^{(1)}_1)>0\right)}
\end{multline*}
while the second is
\[
e^{-\v(x_0,\tilde x^{(1)}_1)} \int_{e^{\v(x_0,\tilde x^{(1)}_1)}}^1 F(\tilde x^{(1)}_1, u)\,du\, \mathbf{1}_{\left(\v(x_0,\tilde x^{(1)}_1)\leq 0\right)}\;.
\]
Combining terms we obtain
\begin{equs}
\mathbf{E}\biggl( \sum_{j=1}^{N_1} F(x^{(j)}_1, \theta^{(j)}_1)\,\bigg \vert\, x_0, \tilde x^{(1)}_1\biggr)
& =e^{-\v(x_0,\tilde x^{(1)}_1)} \int_0^{1}F(\tilde x^{(1)}_1, u)\,du\,\mathbf{1}_{\left(\v(x_0,\tilde x^{(1)}_1)\leq 0\right)} \\
&\qquad +     e^{-\v(x_0,\tilde x^{(1)}_1)} \int_0^{1}F(\tilde x^{(1)}_1, u)\,du\,\mathbf{1}_{\left( \v(x_0,\tilde x^{(1)}_1)>0\right)}        \\
&= e^{-\v(x_0,\tilde x^{(1)}_1)}  \int_0^{1} F(\tilde x^{(1)}_1, u)\,du\;,
\end{equs}
which establishes \eqref{onestep1}.
A similar (but simpler) argument proves the result for $\mathbf{E}$ replaced by $\mathbf{E}^1$.
\end{proof}

With Lemma~\ref{lemma:onestep} in hand we are ready to prove Theorem~\ref{theorem:mean}.
\begin{proof}[of Theorem~\ref{theorem:mean}]
First notice that by Lemma~\eqref{lemma:onestep}, we have the identity
\begin{equation}\label{onestep}
\mathbf{E} \sum_{j=1}^{N_1} F(x^{(j)}_1, \theta^{(j)}_1)  =  \mathbf{E}\biggl( e^{-\v(x_0, y_1)}  \int_0^{1} F(y_1, u)\,du\biggr)\;,
\end{equation}
for any function $F(x,\theta)$.
In particular, the required identity holds for $k=1$ by setting $F(x,\theta) = f(x)$.

Now we proceed by induction, assuming that expression \eqref{mean} holds for step $k-1$ and prove that the relation also holds at $k$.
Notice that, by the Markov property of the process,
\[
\mathbf{E} \sum_{j=1}^{N_{k}} f(x^{(j)}_{k})
= \mathbf{E}\Biggl(\sum_{j=1}^{N_1}\mathbf{E}_{x^{(j)}_1,\theta^{(j)}_1}\biggl(\sum_{j=1}^{N_{k-1}} f(x^{(j)}_{k-1})\biggr)\Biggr).
\]
Setting
\[
F(x,\theta) = \mathbf{E}_{x,\theta}\sum_{j=1}^{N_{k-1}} f(x^{(j)}_{k-1})
\]
in expression \eqref{onestep}, we obtain
\begin{align*}
\mathbf{E} \sum_{j=1}^{N_{k}} f(x^{(j)}_{k}) &=   \mathbf{E}\sum_{j=1}^{N_1} F(x^{(j)}_1,\theta^{(j)}_1)\\
& =  \mathbf{E}\biggl( e^{-\v(x_0, y_1)}  \int_0^{1} \mathbf{E}_{y_1, u}\biggl(\sum_{j=1}^{N_{k-1}} f(x^{(j)}_{k-1})\biggr)\,du\biggr)\;.
\end{align*}
According to the rule for generating the initial ticket we have that
\[
\int_0^{1} \mathbf{E}_{y_1, u} \biggl( \sum_{j=1}^{N_{k-1}} f(x^{(j)}_{k-1})\biggr) \,du 
=  \mathbf{E}_{y_1} \sum_{j=1}^{N_{k-1}} f(x^{(j)}_{k-1})
\]
and we have therefore shown that 
\[
\mathbf{E} \sum_{j=1}^{N_{k}} f(x^{(j)}_{k}) 
=  \mathbf{E}\biggl( e^{-\v(x_0, y_1)}  \mathbf{E}_{y_1}\Bigl(\sum_{j=1}^{N_{k-1}} f(x^{(j)}_{k-1})\Bigr)\biggr)\;.
\]
We can conclude, by our induction hypothesis, that
\begin{equs}
\mathbf{E}\sum_{j=1}^{N_{k}} f(x^{(j)}_{k})
&=
 \mathbf{E}\biggl( e^{-\v(x_0, y_1)}  \mathbf{E}_{y_1}\Bigl(  f(y_{k-1})e^{-\sum_{k=0}^{t-2} \v(y_{k},y_{k+1})}\Bigr)\biggr)\\
& =  \mathbf{E}\biggl( f(y_t) \exp\Bigl({-\sum_{k=0}^{t-1} \v(y_{k},y_{k+1})}\Bigr) \biggr)\;,
\end{equs}
and the proof is complete.  A similar argument proves the result for $\mathbf{E}$ replaced by $\mathbf{E}^1$.
\end{proof}

The second main result of this section compares the variance of the estimators generated by Algorithms \ref{dmc} and \ref{tdmc}.  We have that
\begin{theorem}\label{theorem:varcompare}
For  all bounded functions $f$, one has the inequality
\[
\mathbf{var} \widehat f_t \leq \mathbf{var}^1 \widehat  f_t  \;.
\]
The inequality is strict when $f$ is strictly positive.
\end{theorem}

The following straightforward observation is an important tool in  comparing Algorithms \ref{dmc} and \ref{tdmc}.
\begin{lemma}\label{tdmc2dmc} After  replacing
the rules
\[
\theta^{(j,1)}_{{k+1}} = \frac{\theta^{(j)}_{{k}}}{P^{(j)}} \quad \text{and}\quad \theta^{(j,i)}_{{k+1}}\sim  \mathcal{U}((P^{(j)})^{-1},1) \quad\text{for $i>1$}\;,
\]
in Step 4 of Algorithm~\ref{tdmc} by 
\begin{equ}[e:resampl]
\theta^{(j,i)}_{{k+1}}\sim  \mathcal{U}(0,1)\quad \text{ for all $i$}\;,
\end{equ}
the process generated by Algorithm~\ref{tdmc} becomes identical in law to the process generated by Algorithm~\ref{dmc}.
\end{lemma}

\begin{remark}
Note that the resampling of the tickets in \eref{e:resampl} applies to \textit{all} particles, including $i=1$.
\end{remark}

As with the proof of Theorem~\ref{theorem:mean}, the proof of Theorem~\ref{theorem:varcompare} is inductive.  The most cumbersome component of that argument is obtaining a comparison  of cross terms that appear when expanding the variance of the estimators produced by Algorithms \ref{dmc} and \ref{tdmc}.  This is encapsulated in the following lemma.
\begin{lemma}\label{lemma:varcompare}
 Let $F$ be any non-negative bounded function of $\mathbb{R}^d\times \mathbb{R}$, which is decreasing in its second argument.  Then
\begin{equ}
\mathbf{E}\biggl( \sum_{j=1}^{N_{1}} \sum_{i\neq j}^{N_{1}} F(x^{(j)}_{1},\theta^{(j)}_{1})F(x^{(i)}_{1},\theta^{(i)}_{1})\biggr)
\leq \mathbf{E}^1\biggl( \sum_{j=1}^{N_{1}} \sum_{i\neq j}^{N_{1}} F(x^{(j)}_{1},\theta^{(j)}_{1})F(x^{(i)}_{1},\theta^{(i)}_{1})\biggr)\;,
\end{equ}
and the bound is strict when, for each $x$, $F(x,\theta)$ is a strictly decreasing function of $\theta$.
\end{lemma}


\begin{proof}
First note that $x^{(j)}_1= \tilde x^{(1)}_1$ for each $j$.  Furthermore, conditioned on $x_0$, $\tilde x^{(1)}_1$, and $N_1$, the tickets are independent of each other and
  for $j\geq 2$ they are identically distributed (for Algorithm~\ref{dmc} they are identically distributed for $j\geq1$).
These facts imply that, for the new scheme,
\begin{multline}\label{new1}
\mathbf{E}\biggl( \sum_{j=1}^{N_1} \sum_{i\neq j}^{N_1} F({\tilde x^{(1)}_1,\theta^{(j)}_1})F({\tilde x^{(1)}_1,\theta^{(i)}_{1}})\,\bigg\vert  \, \tilde x^{(1)}_1\biggr) = \\
\mathbf{E}\left(2\, \mathbf{1}_{\left(N_1>1\right)}(N_1-1)\,\big\vert\, \tilde x^{(1)}_1\right) \mathbf{E}\left( F({\tilde x^{(1)}_1,\theta^{(1)}_1})\,\big\vert\, \tilde x^{(1)}_1\right)  \mathbf{E}\left(F({\tilde x^{(1)}_1,\theta^{(2)}_1})\,\big\vert\, \tilde x^{(1)}_1\right)\\
 + \mathbf{E}\left(\mathbf{1}_{\left(N_1>1\right)}(N_1-1)(N_1-2)\,\big\vert\, \tilde x^{(1)}_1\right) \mathbf{E}\left(F({\tilde x^{(1)}_1,\theta^{(2)}_1}) \,\big\vert\, \tilde x^{(1)}_1\right)^2.
\end{multline}
For Algorithm~\ref{dmc}  the tickets $\theta^{(j)}_1$ are i.i.d. conditioned on $\tilde x^{(1)}_1$ so that
\begin{multline}\label{classical1}
\mathbf{E}^1\biggl( \sum_{j=1}^{N_1} \sum_{i\neq j}^{N_1} F({\tilde x^{(1)}_1,\theta^{(j)}_1})F({\tilde x^{(1)}_1,\theta^{(i)}_1})\,\bigg\vert  \, \tilde x^{(1)}_1\biggr) =\\
\mathbf{E}^1\left(N_1(N_1-1) \,\big\vert\, \tilde x^{(1)}_1\right)\mathbf{E}^1\left( F(\tilde x^{(1)}_1,\theta^{(1)}_1)\,\big\vert\, \tilde x^{(1)}_1 \right)^2.
\end{multline}

Let
\[
P = e^{-\v(x_0,\tilde x^{(1)}_{1})}\;.
\]
Note that on the set $\left\{P\leq 1\right\}$, we have $N_{1}\leq 1$, so that expressions \eqref{new1} and \eqref{classical1}  vanish.
On the set $\left\{P> 1\right\}$, we have 
\begin{equs}
\mathbf{E}\left(F({\tilde x^{(1)}_{1},\theta^{(2)}_{1}})\,\big\vert\, \tilde x^{(1)}_{1}\right)
 &=  \frac{1}{1-P^{-1}}\int_{P^{-1}}^{1}  F({\tilde x^{(1)}_{1}, u})
 \,du\;,\\
\mathbf{E}\left(F({\tilde x^{(1)}_{1},\theta^{(1)}_{1}})\,\big\vert\, \tilde x^{(1)}_{1}\right)
 &=  P \int_{0}^{P^{-1}}  F({\tilde x^{(1)}_{1}, u})
\, du\;,
\end{equs}
and
\[
\mathbf{E}^1\left( F(\tilde x^{(1)}_{1},\theta^{(1)}_{1})\,\big\vert\, \tilde x^{(1)}_{1} \right) =  \int_0^{1} F(\tilde x^{(1)}_{1}, u)\,du\;,
\]
so that
\begin{equ}
\mathbf{E}^1\left(F({\tilde x^{(1)}_{1},\theta^{(1)}_{1}})\,\big\vert\, \tilde x^{(1)}_{1} \right)= 
\frac{1}{P}\mathbf{E}\left(F({\tilde x^{(1)}_{1},\theta^{(1)}_{1}})\,\big\vert\, \tilde x^{(1)}_{1}\right)
+ \frac{(P-1)}{P}
\mathbf{E}\left(F({\tilde x^{(1)}_{1},\theta^{(2)}_{1}})\,\big\vert\, \tilde x^{(1)}_{1}\right)\;.
\end{equ}
In other words, if we set
\[
A = \mathbf{E}\left(F({\tilde x^{(1)}_{1},\theta^{(2)}_{1}})\,\big\vert\, \tilde x^{(1)}_{1}\right)
\]
and
\[
B = \mathbf{E}\left(F({\tilde x^{(1)}_{1},\theta^{(1)}_{1}})\,\big\vert\, \tilde x^{(1)}_{1}\right)-\mathbf{E}\left(F({\tilde x^{(1)}_{1},\theta^{(2)}_{1}})\,\big\vert\, \tilde x^{(1)}_{1}\right) ,
\]
we have the identity
\begin{equation*}
\mathbf{E}^1\left(F({\tilde x^{(1)}_{1},\theta^{(1)}_{1}})\,\big\vert\, \tilde x^{(1)}_{1} \right)= 
A + \frac{1}{P}B.
\end{equation*}
In Algorithm~\ref{tdmc}, $\theta^{(1)}_{1} \leq \theta^{(2)}_{1}$ so that, if
$F$ is a decreasing function of $\theta$, then 
\[
B\geq 0.
\]

Now note that the variable $N_{1}$ is the same for both algorithms and is given
explicitly by the formula
\begin{multline*}
N_{1} = \mathbf{1}_{\left(\theta^{(1)}_{1}\leq e^{-\v(x_0,\tilde x^{(1)}_{1})}\leq 1\right)} 
+ \mathbf{1}_{\left(\v(x_0,\tilde x^{(1)}_{1})<0\right)}\\
\times \left( 
\left\lfloor e^{-\v(x_0,\tilde x^{(1)}_{1})} \right\rfloor +
\mathbf{1}_{\left( U < e^{-\v(x_0,\tilde x^{(1)}_{1})}- \left\lfloor
e^{-\v(x_0,\tilde x^{(1)}_{1})} \right\rfloor\right)}\right)\;,
\end{multline*}
where $U$ is a uniform random variable in $[0,1]$. 
Using this formula we can compute the expectations involving $N_{1}$ explicitly.  Defining
\[
R = P - \lfloor P \rfloor,
\]
 on the event $\left\{P> 1\right\}$  we have that
\begin{equs}
\mathbf{E}^1\left(N_{1}(N_{1}-1) \,\vert\, \tilde x^{(1)}_{1}\right) &= (P-R)(P-1+R)\;,\\
\mathbf{E}\left(\mathbf{1}_{\left(N_{1}>1\right)}(N_{1}-1)\,\vert\, \tilde x^{(1)}_{1}\right) &= P-1\;,
\end{equs}
and
\[
\mathbf{E}\left(\mathbf{1}_{\left(N_{1}>1\right)}(N_{1}-1)(N_{1}-2)\,\vert\, \tilde x^{(1)}_{1}\right) 
= (P-1-R)(P-2+R).
\]

The difference of terms \eqref{classical1} and \eqref{new1}, which vanishes on $P<2$, can now be written as
\begin{equation*}
2(P-1)AB   - 2(P-R)(P-1+R)AB \frac{1}{P} - (P-R)(P-1+R)B^2\frac{1}{P^2}
\end{equation*}
on the set $\{P\geq 2\}$.  Recalling that $B\geq 0$ (we can drop the last term) and rearranging we see that this expression is bounded above by
\begin{equation}\label{squares}
2\frac{AB}{P}\left(P(P-1)   - (P-R)(P-1+R)\right).
\end{equation}
Note that since $0\leq R\leq 1$ we have $P-1\leq P-1+R \leq P-R \leq P$.  The difference in the parenthesis in \eqref{squares} is the difference in the area of two squares with equal perimeter, the second of which (in the sense of the last sequence of inequalities) is closer to square.  The difference is therefore non-positive.  
All bounds are strict whenever, for each $x$,  $F(x,\theta)$ is a strictly decreasing function of $\theta$.
\end{proof}

Finally we complete the proof that the variance of the estimator generated by Algorithm~\ref{tdmc} is bounded above by the variance of the estimator generated by Algorithm~\ref{dmc}.
\begin{proof}[of Theorem~\ref{theorem:varcompare}]
By Theorem~\ref{theorem:mean} we have that
\[
\mathbf{E} \widehat  f_t = \mathbf{E}^1 \widehat f _t\;, 
\]
so it suffices to show that
\begin{equation}\label{varinduct}
\mathbf{E}({\widehat  f _t})^2
 \leq \mathbf{E}^1( \widehat  f _t) ^2\;. 
\end{equation}
Furthermore, since the variance does not change by adding a constant to $f$ we can, and will from now on, 
assume that $f(x) \ge 0$ for every $x$.
In order to prove \eref{varinduct}, we will proceed by induction.  Since the random variables $N_{1}$ and $\tilde x^{(1)}_{1}$ are the same in both schemes, the result is true for
$k=1$.  Now assume that \eqref{varinduct} holds through step $k-1$.  We will show that
\begin{equation}\label{varinduct2}
\mathbf{E}\Bigl(\sum_{j=1}^{N_{{k}}} f(x^{(j)}_{{k}})\Bigr)^2 \leq \mathbf{E}^1\Bigl( \sum_{j=1}^{N_{{k}}} f(x^{(j)}_{{k}})\Bigr)^2\;. 
\end{equation}
Observe that we can write
\[
\sum_{j=1}^{N_{{k}}} f(x^{(j)}_{{k}}) =
 \sum_{j=1}^{N_{1}} \sum_{i=1}^{N^{(j)}_{{1},{k}}} f(x^{(j,i)}_{{1},{k}})\;,
\]
where  we used $N^{(j)}_{{1},{k}}$  to denote the number of particles at time ${k}$ whose ancestor at time ${1}$ is $x^{(j)}_{1}$.  The
descendants of $x^{(j)}_{1}$ are enumerated by $x^{(j,i)}_{{1},{k}}$ for $i=1,2,\dots,N^{(j)}_{{1},{k}}$.  

By the conditional independence of the descendants of the particles at time ${k}$ when
 conditioned on $x_0$, $\tilde x^{(1)}_{1}$, and $N_{1}$,  we have that
\begin{equs}
\mathbf{E}\biggl(\sum_{j=1}^{N_{{k}}} f(x^{(j)}_{{k}})\biggr)^2
&= \mathbf{E}\biggl( \sum_{j=1}^{N_{1}} \mathbf{E}_{x^{(j)}_{1},\theta^{(j)}_{1}} \biggl( \sum_{i=1}^{N_{k-1}} f(x^{(i)}_{k-1})\biggr)^2\biggr)\\
&\quad + \mathbf{E}\Biggl( \sum_{j=1}^{N_{1}} \sum_{i\neq j}^{N_{1}}  \mathbf{E}_{x^{(j)}_{1},\theta^{(j)}_{1}} \biggl(\sum_{l=1}^{N_{k-1}} f(x^{(l)}_{k-1})\biggr)\\ &\qquad \times
\mathbf{E}_{x^{(i)}_{1},\theta^{(i)}_{1}} \biggl( \sum_{l=1}^{N_{k-1}} f(x^{(l)}_{k-1})\biggr)\Biggr)\;.\label{expsquare1}
\end{equs}
An identical expression (with $\mathbf{E}$ replaced by $\mathbf{E}^1$) holds for Algorithm~\ref{dmc}.
Setting
\[
F_1(x,\theta) =  \mathbf{E}_{x,\theta} \sum_{l=1}^{N_{k-1}} f(x^{(l)}_{k-1})\;,
\]
and
\[
F_2(x,\theta) =  \mathbf{E}_{x,\theta} \biggl( \sum_{i=1}^{N_{k-1}} f(x^{(i)}_{k-1})\biggr)^2\;,
\]
we can rewrite the last identity in the case of Algorithm~\ref{tdmc} as
\begin{equ}
\mathbf{E}\biggl(\sum_{j=1}^{N_{{k}}} f(x^{(j)}_{{k}})\biggr)^2
= \mathbf{E} \sum_{j=1}^{N_{1}} F_2({x^{(j)}_{1},\theta^{(j)}_{1}})
+ \mathbf{E} \sum_{j=1}^{N_{1}} \sum_{i\neq j}^{N_{1}}  F_1({x^{(j)}_{1},\theta^{(j)}_{1}} )
F_1({x^{(i)}_{1},\theta^{(i)}_{1}}).
\end{equ}
In the case of Algorithm~\ref{dmc} with
\[
F^0_1(x,\theta) =  \mathbf{E}^1_{x,\theta} \sum_{l=1}^{N_{k-1}} f(x^{(l)}_{k-1})
\]
and
\[
F^0_2(x,\theta) =  \mathbf{E}^1_{x,\theta} \biggl( \sum_{i=1}^{N_{k-1}} f(x^{(i)}_{k-1})\biggr)^2\;,
\]
we have
\begin{equ}
\mathbf{E}^1\biggl(\sum_{j=1}^{N_{{k}}} f(x^{(j)}_{{k}})\biggr)^2
= \mathbf{E}^1 \sum_{j=1}^{N_{1}} F^0_2({x^{(j)}_{1},\theta^{(j)}_{1}})
+ \mathbf{E}^1 \sum_{j=1}^{N_{1}} \sum_{i\neq j}^{N_{1}}  F^0_1({x^{(j)}_{1},\theta^{(j)}_{1}} )
F^0_1({x^{(i)}_{1},\theta^{(i)}_{1}})\;.
\end{equ}

Note that Lemma~\ref{lemma:onestep} implies that
\[
\mathbf{E} \sum_{j=1}^{N_{1}} F_2(x^{(j)}_{1},\theta^{(j)}_{1})  = \mathbf{E}^1 \sum_{j=1}^{N_{1}} F_2(x^{(j)}_{1},\theta^{(j)}_{1}).
\]
Under the rules of Algorithm~\ref{dmc}, conditioned on $x_0$ and $\tilde x^{(1)}_{1}$, the ticket $\theta^{(j)}_{1}$ is independent of $N_{1}$.   So we can write the last equality as
\begin{equation*}
 \mathbf{E} \sum_{j=1}^{N_{1}} F_2(x^{(j)}_{1},\theta^{(j)}_{1}) = \mathbf{E}^1\biggl( \sum_{j=1}^{N_{1}}  \mathbf{E}_{x^{(j)}_{1}}\biggl( F_2(x^{(j)}_{1},\theta^{(j)}_{1}) \biggr) \biggr).
\end{equation*}
By our inductive hypothesis we then have that
\begin{equation*}
 \mathbf{E} \sum_{j=1}^{N_{1}} F_2(x^{(j)}_{1},\theta^{(j)}_{1})\leq
 \mathbf{E}^1\biggl( \sum_{j=1}^{N_{1}}  \mathbf{E}_{x^{(j)}_{1}}\biggl( F^0_2(x^{(j)}_{1},\theta^{(j)}_{1}) \biggr) \biggr)\end{equation*}
Appealing again to the conditional independence of $\theta^{(j)}_{1}$ and $N_{1}$ under Algorithm~\ref{dmc} we have that
\begin{equation*}
 \mathbf{E} \sum_{j=1}^{N_{1}} F_2(x^{(j)}_{1},\theta^{(j)}_{1}) \leq
 \mathbf{E}^1\sum_{j=1}^{N_{1}}  F^0_2(x^{(j)}_{1},\theta^{(j)}_{1})\;.
 \end{equation*}

We now move on to the second term in \eqref{expsquare1}.
It follows from the definition of Algorithm~\ref{tdmc} that the function $F_1$ is strictly decreasing in $\theta$,
so that Lemma~\ref{lemma:varcompare} yields
\begin{equ}
 \mathbf{E}\biggl( \sum_{j=1}^{N_{1}} \sum_{i\neq j}^{N_{1}} F_1({x^{(j)}_{1},\theta^{(j)}_{1}})F_1({x^{(i)}_{1},\theta^{(i)}_{1}})\biggr)
\leq
 \mathbf{E}^1\biggl( \sum_{j=1}^{N_{1}} \sum_{i\neq j}^{N_{1}}F_1({x^{(j)}_{1},\theta^{(j)}_{1}})F_1({x^{(i)}_{1},\theta^{(i)}_{1}})\biggr).
\end{equ}
Under the rules of Algorithm~\ref{dmc}, conditional on $x_0$ and $\tilde x^{(1)}_{1}$, the $\theta^{(1)}_{1}$ are all independent of each other and of $N_{1}$.
Therefore we can integrate over the tickets to obtain
\begin{multline*}
  \mathbf{E}\biggl( \sum_{j=1}^{N_{1}} \sum_{i\neq j}^{N_{1}} F_1({x^{(j)}_{1},\theta^{(j)}_{1}})F_1({x^{(i)}_{1},\theta^{(i)}_{1}})\biggr) \\
\leq
 \mathbf{E}^1\biggl( \sum_{j=1}^{N_{1}} \sum_{i\neq j}^{N_{1}}  \mathbf{E}_{x^{(j)}_{1}} \biggl( F_1({x^{(j)}_{1},\theta^{(j)}_{1}})\biggr)
\mathbf{E}_{x^{(i)}_{1}} \biggl(F_1({x^{(i)}_{1},\theta^{(i)}_{1}}) \biggr)\biggr).
\end{multline*}
By Theorem~\ref{theorem:mean} we have that
\[
  \mathbf{E}_{x^{(j)}_{1}} \sum_{l=1}^{N_{k-1}} f(x^{(l)}_{k-1}) =  \mathbf{E}_{x^{(j)}_{1}}^0 \sum_{l=1}^{N_{k-1}} f(x^{(l)}_{{k-1}})
  \]
  or
  \[
   \mathbf{E}_{x^{(j)}_{1}}  F_1({x^{(j)}_{1},\theta^{(j)}_{1}})=  \mathbf{E}^1_{x^{(j)}_{1}} F^0_1({x^{(j)}_{1},\theta^{(j)}_{1}})  \;.
   \]
  This yields 
  \begin{multline*}
 \mathbf{E}\biggl( \sum_{j=1}^{N_{1}} \sum_{i\neq j}^{N_{1}} F_1({x^{(j)}_{1},\theta^{(j)}_{1}})F_1({x^{(i)}_{1},\theta^{(i)}_{1}})\biggr) \\
\leq
 \mathbf{E}^1\biggl( \sum_{j=1}^{N_{1}} \sum_{i\neq j}^{N_{1}}  \mathbf{E}^1_{x^{(j)}_{1}} \biggl( F^0_1({x^{(j)}_{1},\theta^{(j)}_{1}})\biggr)
\mathbf{E}^1_{x^{(i)}_{1}} \biggl(F^0_1({x^{(i)}_{1},\theta^{(i)}_{1}}) \biggr)\biggr).
\end{multline*}
Reinserting  the tickets we obtain
\begin{equ}
 \mathbf{E}\biggl( \sum_{j=1}^{N_{1}} \sum_{i\neq j}^{N_{1}} F_1({x^{(j)}_{1},\theta^{(j)}_{1}})F_1({x^{(i)}_{1},\theta^{(i)}_{1}})\biggr)
\leq
 \mathbf{E}^1\biggl( \sum_{j=1}^{N_{1}} \sum_{i\neq j}^{N_{1}}F^0_1({x^{(j)}_{1},\theta^{(j)}_{1}})F^0_1({x^{(i)}_{1},\theta^{(i)}_{1}})\biggr)\;,
\end{equ}
which completes the proof.
\end{proof}

\bibliographystyle{./Martin}
\markboth{\sc \refname}{\sc \refname}
\bibliography{./refs}

\begin{thebibliography}{AFRtW06}
\expandafter\ifx\csname url\endcsname\relax
  \def\url#1{\texttt{#1}}\fi
\expandafter\ifx\csname urlprefix\endcsname\relax\def\urlprefix{URL }\fi

\bibitem[AFRtW06]{Allen:FFS:2006}
\textsc{R.~J. Allen}, \textsc{D.~Frenkel}, and \textsc{P.~Rein~ten Wolde}.
\newblock Simulating rare events in equilibrium or non equilibrium stochastic
  systems.
\newblock \emph{J. Chem. Phys.} \textbf{124}, (2006), 024102.

\bibitem[And75]{Anderson1975}
\textsc{J.~Anderson}.
\newblock A random-walk simulation of the {S}chr\"odinger equation:
  $\text{H}^+_3$.
\newblock \emph{J. Chem. Phys.} \textbf{63}, no.~4, (1975), 1499--1503.

\bibitem[Buc04]{Bucklew:2004p3891}
\textsc{J.~Bucklew}.
\newblock \emph{Introduction to Rare Event Simulation}.
\newblock Springer, 2004.

\bibitem[CA80]{CeperleyAlder1980}
\textsc{D.~Ceperley} and \textsc{B.~Alder}.
\newblock Ground state of electron gas by a stochastic method.
\newblock \emph{Phys. Rev. Lett.} \textbf{45}, no.~7, (1980), 566--569.

\bibitem[Cri11]{Crisan2011}
\textsc{D.~Crisan}.
\newblock Discretizing the continuous time filtering problem. order of
  convergence.
\newblock In \textsc{D.~Crisan} and \textsc{B.~Rosovskii}, eds., \emph{The
  Oxford Handbook of Nonlinear Filtering},  572--597. Oxford University Press,
  2011.

\bibitem[DBC98]{DellagoBC1998}
\textsc{C.~Dellago}, \textsc{P.~Bolhuis}, and \textsc{D.~Chandler}.
\newblock Efficient transition path sampling: {A}pplication to lennard-jones
  cluster rearrangements.
\newblock \emph{J. Chem. Phys.} \textbf{108}, no.~22, (1998), 9236--9245.

\bibitem[DD11]{MR2833612}
\textsc{T.~Dean} and \textsc{P.~Dupuis}.
\newblock The design and analysis of a generalized {RESTART}/{DPR} algorithm
  for rare event simulation.
\newblock \emph{Ann. Oper. Res.} \textbf{189}, (2011), 63--102.

\bibitem[dDG05]{defreitas05}
\textsc{N.~{de F}reitas}, \textsc{A.~Doucet}, and \textsc{N.~E. Gordon}.
\newblock \emph{Sequential {M}onte {C}arlo Methods in Practice}.
\newblock Springer, 2005.

\bibitem[DM04]{DelMoral:FK:2011}
\textsc{P.~Del~Moral}.
\newblock \emph{Feynman-{K}ac formulae}.
\newblock Probability and its Applications (New York). Springer-Verlag, New
  York, 2004.
\newblock Genealogical and interacting particle systems with applications.

\bibitem[FS96]{Frenkel1996}
\textsc{D.~Frenkel} and \textsc{B.~Smit}.
\newblock \emph{Understanding Molecular Simulation}.
\newblock Academic Press, 1996.

\bibitem[GS71]{GrimmStorer1971}
\textsc{R.~Grimm} and \textsc{R.~Storer}.
\newblock {M}onte-{C}arlo solution of {S}chr\"odinger's equation.
\newblock \emph{J. Comp. Phys.} \textbf{7}, no.~1, (1971), 134--156.

\bibitem[HM54]{HammersleySIS:1954}
\textsc{J.~Hammersley} and \textsc{K.~Morton}.
\newblock Poor man's {M}onte {C}arlo.
\newblock \emph{J. R. Stat. Soc. B} \textbf{16}, no.~1, (1954), 23--38.

\bibitem[HT99]{HarasztiTownsend1999}
\textsc{Z.~Haraszti} and \textsc{J.~Townsend}.
\newblock The theory of direct probability redistribution and its application
  to rare event simulation.
\newblock \emph{ACM Transactions on Modelling and Computer Simulation}
  \textbf{9}, (1999), 105--140.

\bibitem[HW14]{brownianfan}
\textsc{M.~Hairer} and \textsc{J.~Weare}.
\newblock The {B}rownian fan.
\newblock \emph{Commun. Pure Appl. Math.} \textbf{?}, no.~?, (2014), ??--??

\bibitem[JDMD06]{Johansen:SMCrare:2006}
\textsc{A.~Johansen}, \textsc{P.~Del~Moral}, and \textsc{A.~Doucet}.
\newblock Sequential {M}onte {C}arlo samplers for rare events.
\newblock In \emph{Proceedings of the 6th International Workshop on Rare Event
  Simulation}. Bramberg, 2006.

\bibitem[Kal62]{Kalos1962}
\textsc{M.~Kalos}.
\newblock {M}onte {C}arlo calculations of the ground state of three- and
  four-body nuclei.
\newblock \emph{Phys. Rev.} \textbf{128}, no.~4, (1962), 1791--1795.

\bibitem[KL11]{KolorencL2011}
\textsc{J.~Koloren${\check{\text{c}}}$} and \textsc{M.~Lubos}.
\newblock Applications of quantum {M}onte {C}arlo methods in condensed systems.
\newblock \emph{Rep. Prog. Phys.} \textbf{74}, (2011), 1--28.

\bibitem[Liu02]{liu02}
\textsc{J.~Liu}.
\newblock \emph{{M}onte {C}arlo Strategies in Scientific Computing}.
\newblock Springer, 2002.

\bibitem[Lor63]{Lorenz}
\textsc{E.~N. Lorenz}.
\newblock Deterministic nonperiodic flow.
\newblock \emph{J. Atmos. Sci.} \textbf{20}, (1963), 130--141.

\bibitem[MGG94]{MR1271676}
\textsc{R.~N. Miller}, \textsc{M.~Ghil}, and \textsc{F.~Gauthiez}.
\newblock Advanced data assimilation in strongly nonlinear dynamical systems.
\newblock \emph{J. Atmospheric Sci.} \textbf{51}, no.~8, (1994), 1037--1056.

\bibitem[MTAC12]{chorin}
\textsc{M.~Morzfeld}, \textsc{X.~Tu}, \textsc{E.~Atkins}, and \textsc{A.~J.
  Chorin}.
\newblock A random map implementation of implicit filters.
\newblock \emph{J. Comp. Phys.} \textbf{231}, no.~4, (2012), 2049 -- 2066.

\bibitem[Rou06]{Rousset:PhD:2006}
\textsc{M.~Rousset}.
\newblock \emph{M\'{e}thods Population {M}onte-{C}arlo en Temps Continu pour la
  Physique Num\'{e}rique}.
\newblock Ph.D. thesis, L'Universit\'{e} Paul Sabatier Toulouse III, 2006.

\bibitem[RR55]{RosenbluthSIS:1955}
\textsc{M.~Rosenbluth} and \textsc{A.~Rosenbluth}.
\newblock {M}onte {C}arlo calculation of the average extension of molecular
  chains.
\newblock \emph{J. Chem. Phys.} \textbf{23}, no.~2, (1955), 356--359.

\bibitem[VAVA91]{Villen-Altamirano1991}
\textsc{M.~Villen-Altamirano} and \textsc{J.~Villen-Altamirano}.
\newblock {RESTART}: {A} method method for accelerating rare event simulations.
\newblock \emph{Proc. of the 13th international teletraffic congress, queuing,
  performance and control in ATM} \textbf{9}, (1991), 71--76.

\bibitem[VEW]{EveWeare2012b}
\textsc{E.~Vanden-Eijnden} and \textsc{J.~Weare}.
\newblock Data assimilation in the low noise, accurate observation regime with
  application to the kuroshio current.
\newblock \emph{Submitted} .

\bibitem[VEW14]{EveWeare2012}
\textsc{E.~Vanden-Eijnden} and \textsc{J.~Weare}.
\newblock Rare event simulation for small noise diffusions.
\newblock \emph{Commun. Pure Appl. Math.} \textbf{?}, no.~?, (2014), ??--??

\end{thebibliography}

\end{document}